\documentclass[12pt,reqno]{amsart}

\usepackage{amsmath, amssymb, amsthm}
\usepackage{graphicx, color}
\usepackage{enumerate}

\begin{document}

 \bibliographystyle{plain}
 \newtheorem{theorem}{Theorem}
 \newtheorem{lemma}[theorem]{Lemma}
 \newtheorem{corollary}[theorem]{Corollary}
 \newtheorem{problem}[theorem]{Problem}
 \newtheorem{conjecture}[theorem]{Conjecture}
 \newtheorem{definition}[theorem]{Definition}
 \newtheorem{prop}[theorem]{Proposition}
 \numberwithin{equation}{section}
 \numberwithin{theorem}{section}

 \newcommand{\mo}{~\mathrm{mod}~}
 \newcommand{\mc}{\mathcal}
 \newcommand{\rar}{\rightarrow}
 \newcommand{\Rar}{\Rightarrow}
 \newcommand{\lar}{\leftarrow}
 \newcommand{\lrar}{\leftrightarrow}
 \newcommand{\Lrar}{\Leftrightarrow}
 \newcommand{\zpz}{\mathbb{Z}/p\mathbb{Z}}
 \newcommand{\mbb}{\mathbb}
 \newcommand{\B}{\mc{B}}
 \newcommand{\cc}{\mc{C}}
 \newcommand{\D}{\mc{D}}
 \newcommand{\E}{\mc{E}}
 \newcommand{\F}{\mathbb{F}}
 \newcommand{\G}{\mc{G}}
  \newcommand{\ZG}{\Z (G)}
 \newcommand{\FN}{\F_n}
 \newcommand{\I}{\mc{I}}
 \newcommand{\J}{\mc{J}}
 \newcommand{\M}{\mc{M}}
 \newcommand{\nn}{\mc{N}}
 \newcommand{\qq}{\mc{Q}}
 \newcommand{\PP}{\mc{P}}
 \newcommand{\U}{\mc{U}}
 \newcommand{\X}{\mc{X}}
 \newcommand{\Y}{\mc{Y}}
 \newcommand{\itQ}{\mc{Q}}
 \newcommand{\sgn}{\mathrm{sgn}}
 \newcommand{\C}{\mathbb{C}}
 \newcommand{\R}{\mathbb{R}}
 \newcommand{\T}{\mathbb{T}}
 \newcommand{\N}{\mathbb{N}}
 \newcommand{\Q}{\mathbb{Q}}
 \newcommand{\Z}{\mathbb{Z}}
 \newcommand{\A}{\mathcal{A}}
 \newcommand{\ff}{\mathfrak F}
 \newcommand{\fb}{f_{\beta}}
 \newcommand{\fg}{f_{\gamma}}
 \newcommand{\gb}{g_{\beta}}
 \newcommand{\vphi}{\varphi}
 \newcommand{\whXq}{\widehat{X}_q(0)}
 \newcommand{\Xnn}{g_{n,N}}
 \newcommand{\lf}{\left\lfloor}
 \newcommand{\rf}{\right\rfloor}
 \newcommand{\lQx}{L_Q(x)}
 \newcommand{\lQQ}{\frac{\lQx}{Q}}
 \newcommand{\rQx}{R_Q(x)}
 \newcommand{\rQQ}{\frac{\rQx}{Q}}
 \newcommand{\elQ}{\ell_Q(\alpha )}
 \newcommand{\oa}{\overline{a}}
 \newcommand{\oI}{\overline{I}}
 \newcommand{\dx}{\text{\rm d}x}
 \newcommand{\dy}{\text{\rm d}y}
\newcommand{\cal}[1]{\mathcal{#1}}
\newcommand{\cH}{{\cal H}}
\newcommand{\diam}{\operatorname{diam}}
\newcommand{\bx}{\mathbf{x}}
\newcommand{\Ps}{\varphi}

\parskip=0.5ex

\title[Gaps problems and frequencies of patches]{Gaps problems and frequencies of patches in cut and project sets}
\author[Haynes, Koivusalo, Sadun, Walton]{Alan~Haynes,~
Henna~Koivusalo,~
Lorenzo~Sadun,~
James~Walton}
\thanks{AH, HK, JW: Research supported by EPSRC. \\ \phantom{A..}LS: Research partially supported by NSF Grant DMS-1101326.}

\allowdisplaybreaks

\begin{abstract}
We establish a connection
between gaps problems in Diophantine approximation and the frequency spectrum of patches in cut and project sets with special windows. Our theorems
provide bounds for the number of distinct frequencies of patches of
size $r$, which depend on the precise cut and project sets being used,
and which are almost always less than a power of $\log r$. Furthermore, for a substantial collection of cut and project sets we show that the number of
frequencies of patches of size $r$ remains bounded as $r$ tends to
infinity. The latter result applies to a collection of cut and project
sets of full Hausdorff dimension.
\end{abstract}

\maketitle

\section{Introduction}

\subsection{Overview}\label{SEC.Overview}

The theory of multidimensional Diophantine approximation is concerned with approximating subspaces (or more generally, manifolds) in higher dimensional spaces by rational points. This is closely related to
the theory of cut and project sets, in which
portions of a higher dimensional lattice are projected onto a
lower dimensional set. Cut and project sets are important for a number of reasons. They arise as central objects of study in many problems in number theory and dynamical systems and, as discovered by de Bruijn \cite{deBr1981}, they provide a rich source of aperiodic tilings of Euclidean space. They also give us a convenient mathematical model for producing patterns which occur naturally in physical materials known as quasicrystals \cite{Sene1995, ShecBlecGratCahn1984}. The particular problems which we are studying in this paper have direct consequences to understanding the distribution of distances between molecules and patterns in quasicrystals \cite{Slat1959}, as well as to related problems of understanding the distribution of energy levels in sums of forced harmonic oscillators \cite{BlehHommJiRoedShen2012}.

In Diophantine approximation there are many `gaps' theorems (surveys of which can be found in \cite{FraeHolz1995} and \cite{GeelSimp1993}), which can be divided into two categories. The first category is what we will call Slater-type results, and the second, what we will call Steinhaus-type results. Before stating our main theorems we briefly explain how both of these types of results can be formulated in a natural way as problems about cut and project sets.

In Slater-type results we start with a $\Z^d$-action on a (possibly higher dimensional) torus and count the number of patterns of return times of the orbit of a point to some target region. The prototypical example of this is Slater's Theorem, proved in \cite{Slat1950}, in which we consider return times to an interval of the orbit of a point in $\R/\Z$ under the action of an irrational rotation. Slater's Theorem says that the gap between any two consecutive return times takes one of at most $3$ values. From the point of view of cut and project sets, Slater-type results are results about patch complexity (the number of patterns of a given size). They are directly linked to problems about the cohomology of tiling spaces associated to cut and project sets. This is demonstrated, although without the connection to Diophantine approximation, by a theorem of Julien in \cite{Juli2010}, which shows that the cohomology is finitely generated if and only if the number of patches of size $r$ grows asymptotically as slowly as possible.

In Steinhaus-type results we again consider the orbit of a point in a torus, under a $\Z^d$-action. This time we are interested in understanding the `gaps' between points in the torus in finite pieces of the orbit. The prototype here is the Steinhaus Theorem (first proved by S\'{o}s \cite{Sos1957} and \'{S}wierczkowski \cite{Swie1958}), which says that, for any $\alpha\in\R$ and $N\in\N$, the points $n\alpha~\mathrm{mod}~1$ with $1\le n\le N$, partition the circle into intervals of at most $3$ distinct lengths. By an ergodic theoretic argument, Steinhaus-type problems are equivalent to problems about the number of frequencies of patches of a given size in cut and project sets. {\em This observation will be formalized and developed below, and it will serve as the main gateway in this paper for us to convert back and forth between problems in Diophantine approximation and problems about cut and project sets}.

\subsection{Definitions and results} Let $E$ be a $d$-dimensional subspace of $\R^k$.  After possibly
re-ordering the coordinates of $\R^k$, we can write $E$ as the graph
of a linear function $L: \R^d \to \R^{k-d}$. That is,
\begin{equation} E = \{(x,L(x)): x \in \R^d\}.\end{equation}
$L$ can be expressed as matrix
multiplication: $L(x)=Ax$, or in coordinates,
\[L_i(x) := L(x)_{i} =\sum_{j=1}^d \alpha_{ij} x_j,\]
and we use the $d(k-d)$ matrix elements
$\{\alpha_{ij}\}$ to parametrize the choice of $E$. We also suppose that
$E$ is {\bf totally irrational}, meaning that $E+\Z^k$ is dense in $\R^k$.
Next consider subspaces $F_0, F$ that are complementary to $E$, such
that $F_0 \cap \Z^k = \{0\}$, and let $\pi_1: \R^k \to E$ and
$\pi_2: \R^k \to E$ be projections along $F_0$ and $F$, respectively.
Let $\pi^*: \R^k \to F$ be the projection along $E$. The choice of $F$
is somewhat arbitrary, and we will take $F = \{0\} \times \R^{k-d}$.
Alternate choices found frequently in the literature include $F_0$ itself,
and the orthogonal complement to $E$.

Let $\mc{W}\subset F$ be a compact set that is the closure of its interior, which we will refer to as the {\bf window}. For our
purposes, the window will always be a polytope, usually a parallelotope spanned by
integer vectors. We also define a set $\mc{S}\subset\R^k$, which we call the {\bf strip}, by
$\mc{S} = {\pi^*}^{-1}(\mc{W})$. Finally, for $s \in \R^k/\Z^k$, we define
$$Y_s = \pi_1(\mc{S} \cap (\Z^k + s)).$$
The set $Y_s$ is called a {\bf cut and project set}, or a model set,
and is always a Delone set of finite local complexity.
The significance of the
condition $\Z^k \cap F_0 = \{0\}$ is that each point in $Y_s$ is associated
with a unique point in $\Z^k+s$.
If $s$ is {\bf
regular}, meaning that $(\Z^k + s)$ does not intersect the boundary
of $\mc{S}$, then $Y_s$ is repetitive, meaning that all patterns
appear infinitely often and relatively densely.
One can consider the space of all sets $Y_s$ with $s$ regular, plus
appropriate limiting point patterns.
There is a natural translational action of $E$ on this
tiling space, and the resulting dynamical system is measurably
conjugate to the translational action of $E$ on $\R^k/\Z^k$. The only
invariant measure is Lebesgue measure (since $E$ is totally irrational), and the ergodic theorem relates
statistical properties of any $Y_s$ to volumes in $\R^k$, or to volumes in
$F$. Since these statistical properties are independent of $s$, we
henceforth suppress the subscript and denote our cut and project set as $Y$. See Figure 1.

\begin{figure}[h]\label{fig:cut-and-project}
\centering
\def\svgwidth{0.55\columnwidth}
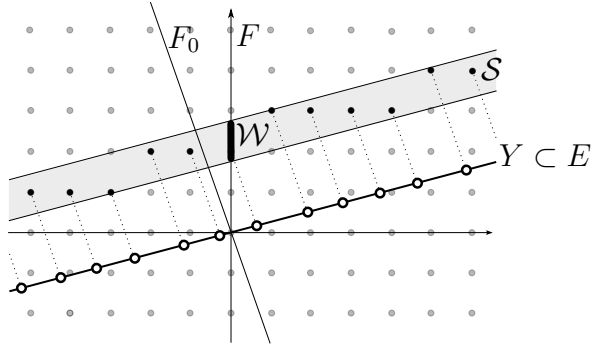
\caption{Definition of the cut and project set $Y$.}
\end{figure}

In what follows we will work with two definitions of patches of size
$r$ in $Y$, illustrated in Figure 2. Assume that we are given a bounded convex set $\Omega\subseteq E$
which contains a neighborhood of $0$ in $E$. For $y\in Y$ and $r\ge 0$
define $P_1(y,r)$, the {\bf type $\mathbf{1}$ patch} of size $r$ at $y$, by
\[P_1(y,r):=\{y'\in Y:y'-y\in r\Omega\}.\] Writing $\tilde{y}$ for the
(unique) point in $\mc{S} \cap (\Z^k + s)$ with $\pi_1(\tilde y)=y$,
we define $P_2(y,r)$, the {\bf type $\mathbf{2}$
patch} of size $r$ at $y$, by
\[P_2(y,r):=\{ y' \in Y: \pi_2(\tilde y'-\tilde y) \in r\Omega\}.\]
In other words, a type 1 patch consists of all points of $Y$
in a certain neighborhood of $y$ in $E$, while a type 2 patch consists of
the projections of all points of $\mc{S}$ whose first $d$ coordinates are
in a certain neighborhood of the first $d$ coordinates of $\tilde y$.
Type $1$ patches are more natural from the point of view of
working within $E$, but the behavior of type $2$ patches is more
closely tied to the Diophantine properties of $L$.

\begin{figure}[h]\label{fig:patches}
\centering
\def\svgwidth{.8\columnwidth}
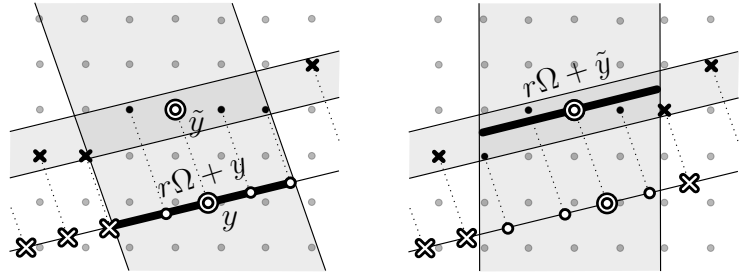
\caption{Definitions of patches of types $1$ and $2$. Points in the doubly shaded area correspond to the lifted patches. The (lifted) points of $Y$ which are not contained in the radius $r$ patch centered at $y$ are indicated with a cross. Notice that near the boundary of the patch the definitions differ.}
\end{figure}

Since the window $\mc{W}$ is compact,
type $1$ and $2$ patches of size $r$ at $y$ differ at
most within a constant neighborhood of the boundary of $r\Omega$.
It turns out to be substantially easier from a technical point
of view to work with type $2$ patches. The strategy of this paper is to
prove results for type $2$ patches and then estimate the error when converting
to type $1$ patches. The reader should note that our definitions
are slightly different than those appearing in the closely related
work \cite{Juli2010} of Antoine Julien, although the pointed patches
which he defines are similar to our type $2$ patches.

For $i=1$ or $2$ and $y_1,y_2\in Y$, we say that $P_i(y_1,r)$ and
$P_i(y_2,r)$ are equivalent if \[P_i(y_1,r)=P_i(y_2,r)+y_1-y_2.\] This
defines an equivalence relation on the collection of type $i$
patches of size $r$. We denote the equivalence class of such a patch
by $\mc{P}_i(y,r)$.

For each equivalence class $\mc{P}_i=\mc{P}_i(y,r)$, we define
$\xi_{\mc{P}_i}$, the {\bf frequency of $\mathbf{\mc{P}_i}$}, by
\[\xi_{\mc{P}_i}:=\lim_{R\rar\infty}\frac{\#\{y'\in Y:|y'|\le
R,~\mc{P}_i(y',r)=\mc{P}_i(y,r)\}}{\#\{y'\in Y:|y'|\le R\}}.\] It is
not difficult to show that, in our setup, the limit defining
$\xi_{\mc{P}_i}$ always exists. Finally, for each $r\ge 0$, we define
$\xi_i (r)$ to be the set of distinct values taken by
$\xi_{\mc{P}_i}$, as $\mc{P}_i$ runs over all equivalence classes of
type $i$ patches of size $r$. We call $\xi_i (r)$ the {\bf frequency
spectrum} of type $i$ patches of size $r$.

Without some hypotheses on the window
$\mc{W}$, the frequency spectrum can behave quite erratically. On the
other hand, in many applications it is typical to choose $\mc{W}$ to
be a nice region, for example a parallelotope that is the image under $\pi^*$ of the
convex hull of
a collection of vectors in $\Z^k$. The
machinery which we will develop in this paper is flexible enough to
say something for general regions $\mc{W}$. However, in order to prove
stronger results, which still apply to interesting known examples, we
specialize to regions $\mc{W}$ which are
parallelotopes with vertices in $F \cap \Z^k \simeq \Z^{k-d}$.

Our first two results give upper bounds for the frequency spectrum, which rely on various Diophantine approximation hypotheses on the subspaces $E$. Note, for comparison, that the number of patches of a given size always grows at least as fast as a constant times $r^d$.
\begin{theorem}\label{THM.FreqsUppBds1}
There is a set of linear maps $L:\R^d \to \R^{k-d}$ of Hausdorff dimension
$d(k-d)$ with the property that, if $E$ is the graph of $L$, and if $\mc{W}$ is a
$(k-d)$-dimensional parallelotope generated by vectors in $F\cap
\Z^k$, then there exists a constant $C$ such that
\[\#\xi_1(r),\#\xi_2(r)\le C ~\text{ for all }~ r\ge 0.\]
\end{theorem}
The Diophantine approximation condition that we will use to prove Theorem \ref{THM.FreqsUppBds1} is a `badly approximable' condition, defined below, which applies to a set of numbers of full Hausdorff dimension but zero Lebesgue measure. If we instead use a condition which applies almost everywhere (in the sense of Lebesgue measure), we are able to obtain the following result.
\begin{theorem}\label{THM.FreqsUppBds2}
There is a set of linear maps $L:\R^d \to \R^{k-d}$ of full Lebesgue measure with the
property that, if $E$ is the graph of $L$, and if
$\mc{W}$ is a $(k-d)$-dimensional
parallelotope generated by vectors in $F\cap \Z^k$, then for any
$\epsilon>0$ and for all sufficiently large $r$, \[\#\xi_2(r)\le (\log
r)^{(1+\epsilon)(d+1)(k-d)}.\] Furthermore if $\Omega$ is the set of points
in $E$ whose first $d$ coordinates lie
in a hypercube centered at the origin
(e.g., points with $-1<x_i<1$ for $i=1,2, \ldots, d$)
then, for all sufficiently large $r$, \[\#\xi_1(r)\le
(\log r)^{(1+\epsilon)(d+1)(k-d)}.\] \end{theorem}
It is worth mentioning that, in the special case when $d=1$, our proofs of Theorems \ref{THM.FreqsUppBds1} and \ref{THM.FreqsUppBds2} can be simplified and, regardless of Diophantine properties, we can obtain uniform bounds for $\#\xi_1(r)$ and $\#\xi_2(r)$. On the other hand we will prove, using a modification of an argument from \cite{BlehHommJiRoedShen2012}, that for any $k$ and $d$ with $d\ge 2$, there is a large collection of subspaces for which the upper bounds in our previous theorems cannot be obtained.
\begin{theorem} \label{THM.FreqsLowBds}
If $d\ge 2$ then there is a set of linear maps $L:\R^d \to \R^{k-d}$ of Hausdorff dimension
$d(k-d)+(1/d-1)$, with the property that, if $E$ is the graph of $L$, we can choose a region $\Omega$ which is a polytope, and for which there exists an $\epsilon>0$ such that
\[\limsup_{r\rar\infty}\frac{\# \xi_2(r)}{r^\epsilon}=\infty.\]
\end{theorem}

Finally, as an interesting sidenote, which complements the results above, we point out that the connection described in the last paragraph of Section \ref{SEC.Overview} can be used give a `topological' proof of the Steinhaus Theorem (as well as other classical Steinhaus-type results). We present such a proof in the appendix to this paper.

\subsection{Summary of notation} For $x\in\R,~\{x\}$ denotes the fractional part of $x$ and
$\| x \|$ denotes the distance from $x$ to the nearest integer. For
$x\in\R^m$ (with $m=d$, $k$, or $k-d$),
$|x|=\max\{|x_1|,\ldots ,|x_m|\}$ and
$\|x\|=\max\{\|x_1\|,\ldots ,\|x_m\|\}.$
We use $\dim (A)$ to
denote the Hausdorff dimension of a set $A$. For a measurable subset
$K\subseteq\R^m$, we write $|K|$ to denote its $m$-dimensional
Lebesgue measure, unless otherwise specified.

\section{Results from Diophantine
approximation}\label{SEC.Diophantine}
Let
$L$ be a linear map $\R^d \to \R^{k-d}$ given by a matrix with entries
$\{\alpha_{ij}\} \in [0,1)^{d(k-d)}$.
Given a function $\psi:\N\rar [0,\infty)$,
let $\mc{E}(\psi)$ denote the collection of points $\alpha\in
[0,1)^{d(k-d)}$ for which the inequality
\[ \| L(n)\| \le \psi(|n|)\] is satisfied by infinitely many
$n\in\Z^d$. Note that $\|L(n)\|$ would not change if we added integers to
the coefficients $\alpha_{ij}$, which is why we can restrict $\alpha$
to $[0,1)^{d(k-d)}$ without losing generality.

For any $N\in\N,$ there exists an $n\in\Z^d$
with $|n|<N$ and
\begin{equation}\label{EQN.DirichLinForms}
\| L(n) \| \le\frac{1}{N^{d/(k-d)}}.
\end{equation}
This is a
multidimensional analogue of Dirichlet's Theorem, which follows from a straightforward application of the pigeonhole principle. It can also be
stated by saying that $\mc{E}(\psi)=[0,1)^{d(k-d)}$ when
\[\psi(m)=m^{-d/(k-d)}.\] We are interested in having an inhomogeneous
version of this result, requiring the values taken by
$\|L(n)-\gamma\|$ to be small, for all choices of
$\gamma\in\R^{k-d}$. This is not quite as straightforward as it may
seem, and in fact extra hypotheses on the numbers $(\alpha_{ij})$ are
necessary in order to obtain such results. However these types of results are dealt with nicely by the following
`transference theorem,' a proof of which can be found in \cite[Chapter
V, Section 4]{Cass1957}.
\begin{theorem}\cite[Section V, Theorem
VI]{Cass1957}.\label{THM.TransferTheorem} Suppose that, for $C,X>0$,
there are no solutions $n\in\Z^k\setminus\{0\}$ to the equations
\[ \|L(n)\|\le C\quad\text{and}\quad |n|\le X.\]
Then for all $\gamma\in\R^{k-d}$, there is an $n\in\Z^k$ satisfying
\[ \|L(n)-\gamma\|\le C_1\quad\text{and}\quad
|n|\le X_1,\] where \[C_1=\frac{1}{2}(h+1)C,\quad
X_1=\frac{1}{2}(h+1)X,\quad \text{and}\quad h=\lfloor
X^{-d}C^{d-k}\rfloor.\]
\end{theorem}

Next we investigate
how much equation (\ref{EQN.DirichLinForms}) can be improved.
Of course, the
answer depends on the choice of $\alpha\in\R^{d(k-d)}.$ It follows
easily from the Borel-Cantelli Lemma that if
\[\sum_{m\in\N}m^{d-1}\psi (m)^{k-d}<\infty,\] then the Lebesgue
measure of $\mc{E}(\psi)$ is $0$. If $\psi$ is monotonic,
then the following theorem indicates
that the converse also holds.
\begin{theorem}[Khintchine-Groshev Theorem]\label{THM.KhinGros} If
$\psi$ is monotonic and if \[\sum_{m\in\N}m^{d-1}\psi
(m)^{k-d}=\infty,\] then the Lebesgue measure of $\mc{E}(\psi)$ is
$1$.
\end{theorem}

A proof of this theorem can be found in
\cite{BereDickVela2006}. As an application, let $\mathcal{B}_{d,k-d}$
denote the collection of numbers $\alpha\in\R^{d(k-d)}$ with the
property that there exists a constant $C=C(\alpha)>0$ such that, for
all nonzero integer vectors $n\in\Z^d$,
\[\|L(n)\|\ge\frac{C}{|n|^{d/(k-d)}}.\] The Khintchine-Groshev
Theorem implies that the Lebesgue measure of $\mathcal{B}_{d,k-d}$ is
$0$. However in terms of Hausdorff dimension these sets are large. It
is a classical result of Jarnik that $\dim \mc{B}_{1,1}=1$, and this
was extended by Wolfgang Schmidt, who showed in \cite[Theorem
2]{Schm1969} that, for any choices of $1\le d<k,$ \[\dim
\mc{B}_{d,k-d}=d(k-d).\] In fact Schmidt proved something stronger,
that the set $\dim \mc{B}_{d,k-d}$ is $1/2$-winning, in the sense of a
game, now known as Schmidt's game, which we describe presently.

In Schmidt's game two players, B and A, take turns choosing closed
balls in $\R^s$. Player B begins by choosing $B_1\subseteq\R^s$ with
positive radius, followed by player A who chooses $A_1\subseteq B_1$,
and so on, resulting in a countably infinite nested sequence of closed
balls
\[B_1\supseteq A_1\supseteq B_2\supseteq\cdots.\]
The radii of the balls are determined by requiring that
\[\mathrm{rad}(A_i)=a\cdot\mathrm{rad}(B_i)\quad\text{and}\quad\mathrm{rad}(B_{i+1})=b\cdot\mathrm{rad}(A_i),\]
where $a$ and $b$ are fixed parameters of the game (denoted in most
other sources by the letters $\alpha$ and $\beta$) satisfying
$0<a,b<1$. The intersection of the balls is always a single point, and
a set $S\subseteq\R^s$ is called $a$-winning if player A always has a
winning strategy for causing this point to lie in $S$, with the
parameters $a$ and $b$, for any choice of $0<b<1$. A set is called
winning if it is $a$-winning for some $a$. Winning sets have the
following properties:
\begin{enumerate}
  \item[(W1)] Winning sets in $\R^s$ are dense and have dimension $s$.
  \item[(W2)] For any $a>0$, a countable intersection of $a$-winning
    sets is $a$-winning.
  \item[(W3)] The image of an $a$-winning set under a bi-Lipschitz map
    is $a'$-winning, for some $a'>0$ which depends on $a$ and the
    Lipschitz constant associated with the map.
\end{enumerate}
Property (W2) is particularly desirable, since it is not true in
general that a countable intersection of sets with full Hausdorff
dimension will also have full Hausdorff dimension. However, note that
the properties above do not imply that a countable intersection of
bi-Lipschitz images of $a$-winning sets is winning, unless the
Lipschitz constants of the maps are uniformly bounded. For this
reason, in our proof of Theorem \ref{THM.FreqsUppBds1}, we will need
to use a variation of $a$-winning with a property that is slightly
more robust than (W3).

The variation that we will use is something known as hyperplane
absolute winning (HAW). HAW was introduced in
\cite{BrodFishKleiReicWeis2012} as part of a generalization of another
type of winning, absolute winning, which had previously been
considered by McMullen in \cite{McMu2010}. HAW is defined by a game
called the $(s-1)$-dimensional $b$-absolute game in $\R^s$. Here $b$
is a number with $0<b<1/3$. The game is played by two players, B and
A, who again take turns choosing sets in $\R^s.$ Player B begins by
choosing a closed ball $B_1\subseteq\R^s$ of radius
$\delta_1>0$. Next, player A chooses $A_1$ to be an
$\epsilon_1$-neighborhood of any affine hyperplane in $\R^s$ (i.e. a
translate of a subspace of dimension $s-1$) , where $\epsilon_1$ is
any number satisfying $0<\epsilon_1\le b\delta_1$. Player B continues
by choosing a closed ball $B_2$ with radius $\delta_2\ge b\delta_1$,
which is contained in $B_1\setminus A_1$. The game continues in this
way, and a set $S\subseteq\R^s$ is said to be $(s-1)$-dimensionally
$b$-absolute winning if player A always has a strategy for ensuring
that
\[\bigcap_{i=1}^\infty B_i\cap S\not=\emptyset.\]
A set $S$ is called HAW if it is $b$-absolute winning for all
$0<b<1/3$. It is proved in \cite{BrodFishKleiReicWeis2012} that HAW
sets have the following properties:
\begin{enumerate}
  \item[(H1)] HAW sets are $a$-winning for all $a<1/2$.
  \item[(H2)] A countable intersection of HAW sets is HAW.
  \item[(H3)] The image of an HAW set in $\R^s$ under a $C^1$ diffeomorphism of $\R^s$ is also HAW.
\end{enumerate}
Of course, (H1) implies that HAW sets are dense and have full Hausdorff dimension. In our proof below we will use (H1)-(H3) together with the following recent result of Broderick, Fishman, and Simmons.
\begin{theorem}\cite[Theorem 1.3]{BrodFishSimm2013}.\label{THM.HAWForBad}
For every choice of $k>d\ge 1$ the set $\mc{B}_{d,k-d}$ is HAW.
\end{theorem}
Finally, we will need the following property of HAW sets.
\begin{lemma}\label{LEM.HAWDirectSum}
Suppose that $1\le d<k$ and that
\[\R^s=\R^d\oplus \R^{d(k-d-1)}.\]
If $S$ is an HAW set in $\R^d$ then, with reference to this direct sum decomposition, the set
\[S\oplus \R^{d(k-d-1)}\]
is an HAW set in $\R^s$.
\end{lemma}
\begin{proof}
Consider the $(s-1)$-dimensional $b$-absolute game in $\R^s$, and let $\rho:\R^s\rar\R^d$ be the projection onto the first factor in the direct sum decomposition above.

First player B chooses $\delta_1$ and $B_1$. Let $B_1'=\rho (B_1)$,
and note that $B_1'$ has radius $\delta_1$ in $\R^d$. Choose
$\epsilon_1$ and $A_1'\subseteq\R^d$ according to a winning strategy
for $S$, for the $(d-1)$-dimensional $b$-absolute game. Then player A
can take $A_1=\rho^{-1}(A_1')$, which is an $\epsilon_1$ neighborhood
of an affine hyperplane in $\R^s$.

At the next step player B chooses $B_2\subseteq B_1\setminus A_1$ with
radius $\delta_2$. The image under $\rho$ of $B_2$ is a closed ball
$B_2'\in B_1'\setminus A_1'$ with radius $\delta_2$ in
$\R^d$. Therefore player A can choose $\epsilon_2$ and $A_2'$
according to a winning strategy for $S$ and then take
$A_2=\rho^{-1}(A_2')$. Continuing in this way, we construct a sequence
of sets $B_i,B_i',A_i,$ and $A_i'$ satisfying
\[B_{i+1}\subseteq B_i,\quad\rho(B_i)=B_i',\quad A_i=A_i'\oplus \R^{d(k-d-1)},\]
and
\[\bigcap_{i=1}^\infty B_i'\cap S\not=\emptyset.\]
This guarantees that
\[\bigcap_{i=1}^\infty B_i\cap \left(S\oplus \R^{d(k-d-1)}\right)\not=\emptyset,\]
and since this holds for all $0<b<1/3$, the proof is complete.
\end{proof}

\section{Regular points and the frequency spectrum}
In this section we show how the frequency spectrum of a cut and
project set is determined by the set of volumes of subregions of the
window $\mc{W}$ obtained by a $\Z^d$-action on the
boundary. This allows us to translate problems about the frequency
spectrum into the language of Diophantine approximation.

As above, let $F=\{0\} \times \R^{k-d}$, let $\mc{W} \subset F$, and
consider the strip $\mc{S} = {\pi^*}^{-1}(\mc{W})$.  There is a
natural action of $\Z^k$ on $F$, given by
\[n.w=\pi^*(n)+w = w + (0,n_2-L(n_1)),\]
for $n=(n_1,n_2)\in\Z^k = \Z^d \times
\Z^{k-d}$ and $w\in F$. For each $r\ge 0$ we define the
{\bf $\mathbf{r}$-singular points of type $\mathbf{1}$} by
\[\mathrm{sing}_1(r):=\mc{W}\cap\left((-\pi_1^{-1}(r\Omega)\cap\Z^k). \partial\mc{W}\right),\]
and, similarly, the {\bf $\mathbf{r}$-singular points of type $\mathbf{2}$} by
\[\mathrm{sing}_2(r):=\mc{W}\cap\left((-\pi_2^{-1}(r\Omega)\cap\Z^k). \partial\mc{W}\right).\]
Then, for $i=1$ or $2$ we define the {\bf $\mathbf{r}$-regular points of
type $\mathbf{i}$} by
\[\mathrm{reg}_i(r):=\mc{W}\setminus\mathrm{sing}_i(r).\]
We begin with the following observation.
\begin{lemma}\label{LEM.FreqsVols1}
  For $i=1$ or $2$, if $E$ is totally irrational then every element of $\xi_i
  (r)$ is a sum of volumes of connected components of
  $\mathrm{reg}_i(r)$, divided by the $(k-d)$-dimensional volume of
  $\mc{W}$.
\end{lemma}
\begin{proof}
To each point $y \in Y$, we associate a point $y^* = \pi^*(\tilde y)
\in \mc{W}$. This point determines the pattern around $y$, as follows.
Each point $y' \in Y$ lifts to a point $\tilde y' = \tilde y + n$.
But such a point is in $\mc{S}$ if and only if $\pi^*(\tilde y')
= n.y^*$ lies in $\mc{W}$. As we vary $y^*$, the pattern around $y$ can
only change when some $n.y^*$ passes through $\partial \mc{W}$, that is when
$y^*$ passes from one connected component of $\mathrm{reg}_i(r)$ to another.
The only difference between $i=1$ and $i=2$ is the set of $n$'s
being considered. In both cases, each connected component of $\mathrm{reg}_i(r)$
corresponds to a single equivalence class of patches, although {\em it is possible that two or more components could correspond to the same patch (see Figure 3).}

\begin{figure}[h]\label{fig:disconnected-regions}
\centering\def\svgwidth{.65\columnwidth}
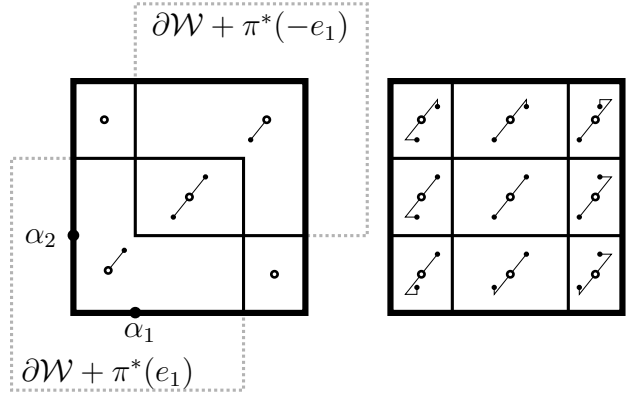
\caption{Example of a $3$-to-$1$ projection with $E$ taken to be the line spanned by $(1,\alpha_1,\alpha_2)$ and $\mc{W} \subset \{0\} \times \mathbb{R}^2$ the unit square. The left-hand figure illustrates the connected components of $\mathrm{reg}_1(r)$ for a small value of $r$. Each component contains a diagram of the corresponding (lifted) type $1$ patch, each centered at its white dot. Note that each $2$-point patch corresponds to a non-convex region and the $1$-point patch correspond to a disconnected region. By comparison, the components of $\mathrm{reg}_2(r)$ (illustrated on the right) are convex and each corresponds to a unique type $2$ patch.}
\end{figure}

Since the action of $E$ on the space of tilings is measurably conjugate
to an irrational toral rotation, and since the only invariant
measure is Lebesgue measure, the ergodic theorem states that
the frequency of each patch is proportional to the
Lebesgue measure of the portion of $\mc{W}$ to which it corresponds. Since
we have normalized our frequencies so that the sum of the frequencies is 1,
the lemma follows.
\end{proof}

We would like to mention that the idea motivating this lemma is very
closely related to Proposition 2.1 of Julien's paper
\cite{Juli2010}. Unfortunately, the Proposition in that paper does not
appear to be correct. Using Julien's notation, it is true that points
lying in the same connected component of $\mathrm{reg}(r)$ give rise
to equivalent pointed patches of size $r$. However the converse is not
true in general, unless $k-d=1$. When $k-d>1$ it can happen that two
pointed patches are identical, but that continuously moving the
distinguished point of one to the other cannot be accomplished without
passing through a region where extra points are temporarily added to
the patch. If it were true that every equivalence class of pointed
patches corresponded uniquely to a connected component of
$\mathrm{reg}(r)$, this would allow for a considerable simplification
of the arguments which follow, as well as obviating the need for
any hypothesis on the shape of $\Omega$ in Theorem \ref{THM.FreqsUppBds2}
and for Lemma \ref{LEM.Reg1->Reg2ForAE} below. However, in general, we
do not even know how to guarantee that there is only a bounded number
of connected components corresponding to each equivalence class of
pointed patch. One of the roles of the Diophantine hypotheses in our
theorems is to limit this bad behavior.

When $i=2$
and $\mc{W}$ is a parallelotope generated by integer vectors, a
stronger conclusion can be drawn.
\begin{lemma}\label{LEM.FreqsVols2}
  If $E$ is totally irrational and if $\mc{W}$ is a parallelotope
  generated by integer vectors, then every element of $\xi_2(r)$ is
  equal to the volume of a single connected component of
  $\mathrm{reg}_2(r)$, divided by the $(k-d)$-dimensional volume of
  $\mc{W}$.
\end{lemma}

\begin{proof}
Following the proof of Lemma \ref{LEM.FreqsVols1}, we must show that, under the additional hypothesis on $\mc{W}$, different connected components of $\mathrm{reg}_2(r)$ correspond to different type $2$ patches of size $r$.

Suppose that $y_1$ and $y_2 \in Y$, and that $P_2(y_1,r)$ is equivalent to
$P_2(y_2,r)$. Imagine varying $y^*$ in
a straight line from $y_1^*$ to $y_2^*$.  In moving $y^*$ from one
connected component to another, the patch $P_2(y,r)$ gains and/or
loses points whenever $y^*$ crosses from one component to another.  We
will show that none of the points of $P_2(y_1,r)$ may be removed in
going from $y_1^*$ to $y_2^*$, and that no points may be added without
removing other points. Combining these observations, no points can be
added or removed, so $y_1^*$ and $y_2^*$ must lie in the same component.

To see that no points may be removed, note that $\mc{W}$ is convex. Thus, for each
$n$ for which $\pi_1(\tilde y_i+n)$ is in $P_2(y_i,r)$,
the set of points $y^*$ satisfying $n.y^* \in \mc{W}$ is convex. Since
$n.y_1^*$ and $n.y_2^*$ are in $\mc{W}$, all points on the line segment connecting them must also be
in $\mc{W}$. Thus  all points $y^*$ on the line segment
correspond to
patches that contain a translate of $P_2(y_i,r)$.

Next notice that, since $\mc{W}$ is a parallelotope generated by integer vectors,
after removing a subset of its boundary it is a (strict) fundamental domain for a sublattice of $\Z^{k-d}$, of some index $I$. This
implies that for each $n_1 \in \Z^d$, and each $\tilde y \in \R^d$,
there are exactly $I$ points $n_2 \in \Z^{k-d}$ such that $\tilde y +
(n_1,n_2) \in \mc{S}$ (with appropriate counting of points on the
boundary of $\mc{S}$). In other words, as we cross a boundary between
connected components, a point is removed from $P_2(y,r)$ for each
point added. We have already shown that no points can be removed, so no
points can be added.
\end{proof}

At first glance, the hypotheses of Lemma \ref{LEM.FreqsVols2} may
seem overly restrictive, since most cut and project tilings do not have
integer parallelotope windows. In so-called `canonical'
cut and project tilings, the window is the image under $\pi^*$ of the unit cube in
$\R^k$, and typically has irrational volume in $F$. However,
the resulting point patterns are often dynamically and combinatorially
equivalent (mutually locally derivable) to point patterns
obtained from integer parallelotope windows.

For example, suppose that $d=k-d=1$.  If $E$ has irrational slope
$\alpha < 1$, and if the window is
\[\mc{W}=\pi^*([0,1]^2)=\{0\} \times [-\alpha, 1],\]
then the set $Y$ codes a
Sturmian sequence. There are two intervals between successive points,
obtained by projecting horizontal and vertical intervals of length 1
onto $E$ along $F_0$. Call these intervals $A$ and $B$, respectively.
To each integer we either associate two intervals $AB$ (meaning $A$, then $B$)
or one interval
$A$, depending on whether $\lfloor(n+1)\alpha + s\rfloor - \lfloor n\alpha + s\rfloor$ equals 1
or 0 (where $\lfloor \cdot \rfloor$ denotes the integer part of a real number). If we delete the points that occur at the beginning of the $B$
intervals (thereby combining each pair $AB$ into a single larger
interval $C$), then we get an equivalent point pattern with exactly
one point per integer, and the spacing between successive points codes
the Sturmian sequence.  This is exactly the cut and project pattern
with window $\{0\} \times [0,1]$. Lemma \ref{LEM.FreqsVols2} does not directly
apply to the canonical cut and project pattern, but counting patterns
in that tiling is essentially equivalent to counting patterns in the modified
pattern, to which Lemma \ref{LEM.FreqsVols2} does apply.

The first part of the proof of Lemma \ref{LEM.FreqsVols2} applies equally
well to the more geometric type $1$ patches.
If the patches associated to $y_1^*$ and $y_2^*$ are equivalent, then
any patch associated to $t y_1^* + (1-t)y_2^*$ must contain all the
points of $P_1(y_i,r)$.  However, the final part of the argument does
not work.  Since $\pi_1(\tilde y+(n_1,n_2))$ depends on both $n_1$ and
$n_2$, some points associated with $n_1$ might have images in $r\Omega
+ y$, while others might not. As we change the (fixed number) of
points associated with $n_1$, points can jump in and out of a patch,
so in going from $y_1^*$ to $y_2^*$, we could gain a point, then have
it leave, leading to the same patch that we started with.

\section{Proof of Theorem \ref{THM.FreqsUppBds1}}
In this section we suppose that $\mc{W}\subseteq F$ is a parallelotope
generated by $k-d$ linearly independent vectors $w_1,\ldots,
w_{k-d}\in F \cap \Z^k$.
We then define an integer linear transformation $B: F \to F$
such that $\mc{W}$ is the image under $B$ of the unit hypercube
$\{0\} \times [0,1]^{k-d}$. Applying the linear transformation $B^{-1}$ converts
the window to the unit hypercube, converts the integer
lattice $\Z^k$ to $\Z^d \times \Lambda$, where $\Lambda$ is a
finite-index extension of $\Z^{k-d}$, and converts $E$ to the graph of
the linear transformation $L' = B^{-1} L$. Let $\{\beta_{ij}\}$
be the matrix elements of $L'$. Since the Jacobian of a linear transformation
is constant, our linear transformation does not affect ratios of volumes, and
so does not change $\xi_1(r)$ or $\xi_2(r)$, once we apply the appropriate
linear transformation to $\Omega$.
We emphasize that $\beta$
depends on both $\mc{W}$ (i.e. on $B$) and on $\alpha$.

In what follows it may be helpful for the reader to keep in mind the
following guiding principles in our proof of Theorem
\ref{THM.FreqsUppBds1}:
\begin{enumerate}
\item[(P1)] Diophantine approximation hypotheses on $\beta$ allow us
  to appeal to Theorem \ref{THM.TransferTheorem} to quantify the
  density modulo $1$ of the values taken by $L'(n)$, for $|n|<r$.
\item[(P2)] Under the hypotheses, the values of $L'(n)$, for
  $|n|<r$, can not accumulate too closely at any point in the
  torus. This guarantees that for any fixed $C>0$ the number of
  hyperplanes of $\mathrm{sing}_2(r+C)$ which intersect any given
  connected component of $\mathrm{reg}_2(r)$ is bounded.
\item[(P3)] By a refinement of the argument used to prove (P2), it
  follows that for each $1 \le i \le k-d$ the number of possible distances
between consecutive
  values of $L_i'(n)~\mathrm{mod}~1$, for $|n|<r$, is uniformly
  bounded.
\item[(P4)] Every choice of parallelotope $\mc{W}$ with integer
  vertices defines a collection of $\alpha$ for which the Diophantine
  hypotheses in (P1) on $\beta$ are satisfied. Each such collection of
  $\alpha$ is the image of an HAW set under a $C^1$ map, therefore the
  intersection of all such sets is HAW and has full Hausdorff
  dimension.
 \end{enumerate}
 First, incorporating principles (P1) and (P2), we demonstrate how
 assumptions on the Diophantine approximation properties of the
 numbers $\beta$ can be used to control the complexity of the sets
 $\mathrm{reg}_1(r).$ We work with the real-valued functions
$L'_i(n) = \sum_{j=1}^d \beta_{ij} n_j$, one at a time.

\begin{lemma}\label{LEM.Reg1->Reg2ForBad}
  With notation as above, suppose that
  $(\beta_{ij})_{j=1}^d\in\mc{B}_{d,1}$ for each $1\le i\le k-d.$ Then
  there exist $c_1,c_2>0$ such that, for all $r>0$, every element of
  $\xi_1(r)$ is a sum of at most $c_1$ volumes of connected components
  of $\mathrm{reg}_2(r+c_2)$, divided by the volume of $\mc{W}$.
\end{lemma}
\begin{proof}
First, it follows from the definitions of type $1$ and type $2$
patches (as well as the convexity of $\Omega$) that we can choose $c_2>0$ so that, for all $y\in Y$ and all sufficiently large $r$,
\begin{equation}\label{EQN.PatchInclusions}
P_2(y,r-c_2)\subseteq P_1(y,r)\subseteq P_2(y,r+c_2).
\end{equation}
This is not completely obvious, but it follows from standard facts about Hausdorff distance and dilations of convex sets.

Therefore, each connected component of $\mathrm{reg}_1(r)$ is contained in a single connected component of $\mathrm{reg}_2(r-c_2),$ and its volume is a sum of volumes
of connected
components of $\mathrm{reg}_2(r+c_2).$
We will show that each component of $\mathrm{reg}_2(r-c_2)$
is the union of at most $c_1$ components of $\mathrm{reg}_2(r+c_2)$,
hence that the volume of each component of $\mathrm{reg}_1(r)$ is
the sum of at most $c_1$ volumes of components of $\mathrm{reg}_2(r+c_2)$.

Let $\Omega'$ be the projection of $\Omega$ on $\R^d$ along $F$. That
is, $\Omega'$ is the set of first $d$ coordinates of the points in $\Omega$.
Let $R_1$ and $R_2$ be closed hypercubes in $\R^d$, each containing a neighborhood of $0$, and satisfying $R_1\subseteq\Omega'\subseteq R_2.$ Also, for $1\le i\le
k-d$, let $H_i$ be the hyperplane in $F$ orthogonal to $e_{d+i}$.
Let
\[\mathrm{sing}_2'(r):=B^{-1}\mathrm{sing}_2(r)\quad\text{and}\quad\mathrm{reg}_2'(r):=B^{-1}\mathrm{reg}_2(r).\] The set $\mathrm{sing}_2'(r)$ is composed of the intersection
of the unit hypercube with translates of the faces of the hypercube
by points of the form $(-n_1,-\lambda) \in \Z^d \times \Lambda$
with $n_1 \in r\Omega'$. The action of $\{0\} \times
\Z^{k-d}$ maps the faces of the unit hypercube into the set
\[\bigcup(H_i + \Z e_{d+i}),\]
so, for the purposes of studying $\mathrm{sing}_2'(r)$, we can restrict our attention to the action
of $\Z^d \times (\Lambda/\Z^{k-d})$.

For any $r>0,$ the set
$\mathrm{sing}_2'(r)$ is therefore given by
\begin{align*}
[0,1]^{k-d}\cap\left(\bigcup_{n\in r\Omega'\cap\Z^d}\bigcup_{i=1}^{k-d}\left(H_i+
L'(n) + \Lambda\right)\right).
\end{align*}
When considering translates of $H_i$, all that matters is the $(d+i)$th
coordinate of the offset. The $(d+i)$th coordinates of $\Lambda$ form
a group of the form $m_i^{-1} \Z$, for some $m_i\in\Z$, so we have that
\[H_i+
L'(n)
+ \Lambda=H_i+(L_i'(n)+m_i^{-1}\Z)e_{d+i}.\]
The assumption that $(\beta_{ij})_{j=1}^d\in\mathcal{B}_{d,1}$ implies
that $(m_i\beta_{ij})_{j=1}^d\in\mathcal{B}_{d,1}$ for each $i$, and
it follows from Theorem \ref{THM.TransferTheorem} that there is a
constant $c_3>0$ such that, for each $1\le i\le k-d,$ the set
\[\{L_i'(n)~\mathrm{mod}~m_i^{-1}:n\in rR_1\cap\Z^k\}\] is
$c_3/r^d$-dense in $\R/m_i^{-1}\Z$. From this we conclude that if $U$
is any connected component of $\mathrm{reg}_2'(r)$ then $U$ is a rectangle
of the form
\[ \{t\in W: \ell_i<t_i<r_i\},\]
with
\[r_i-\ell_i\le\frac{c_3}{m_ir^d}~\text{for each}~1\le i\le k-d.\]
Now observe that the number of connected components of $\mathrm{reg}_2'(r+2c_2)$ which intersect $U$ is equal to
\[\prod_{i=1}^{k-d}\left(1+\#\{n\in ((r+2c_2)\Omega'\setminus r\Omega')\cap\Z^k:L_i'(n)\in (\ell_i,r_i)~\mathrm{mod}~m_i^{-1}\}\right).\]
This is bounded above by
\[\prod_{i=1}^{k-d}\left(1+\#\{n\in (r+2c_2)R_2\cap\Z^k:L_i'(n)\in (\ell_i,r_i)~\mathrm{mod}~m_i^{-1}\}\right),\]
and, again using our hypotheses on $\beta$, we see that the final
quantity is bounded above by a constant $c_1>0$. We have shown that
every connected component of $\mathrm{reg}_2'(r)$ is a union of at most
$c_1$ connected components of $\mathrm{reg}_2'(r+2c_2)$. After applying the linear map $B$ this,
together with the observations in the first paragraph, completes the
proof of the lemma.
\end{proof}
Next, following principle (P3) above, we refine the argument of Lemma
\ref{LEM.Reg1->Reg2ForBad} to count the number of volumes of connected
components of $\mathrm{reg}_2(r)$.
\begin{lemma}\label{LEM.VolumeBdForBAD}
  If $\beta$ satisfies the hypotheses of Lemma
  \ref{LEM.Reg1->Reg2ForBad} then there is a constant $c_4>0$ such
  that, for any $r>0$, there are at most $c_4$ distinct volumes of
  connected components of $\mathrm{reg}_2(r)$.
\end{lemma}
\begin{proof}
  By the argument of Lemma \ref{LEM.Reg1->Reg2ForBad}, there exists a
  constant $c_3>0$ such that every connected component of
  $\mathrm{reg}_2'(r)$ is contained in a cube in $W$ of side
  length no greater than $c_3/r^d$. Suppose that $U$ is a
  connected component of $\mathrm{reg}_2'(r)$ and, using the
  notation in the proof of Lemma \ref{LEM.Reg1->Reg2ForBad}, for each
  $1\le i\le k-d$ choose $n_i,n_i'\in r\Omega'\cap\Z^k$ and
  $a_i,a_i'\in m_i^{-1}\Z$ so that
\begin{equation}\label{EQN.GenericComponent}
U=\prod_{i=1}^{k-d}(a_i+L_i'(n_i),a_i'+L_i'(n_i')).
\end{equation}
For each $i$ we then have that
\begin{equation}\label{EQN.GapsBad1}
(a_i'-a_i)+L_i'(n_i'-n_i)\le\frac{c_3}{r^d},
\end{equation}
and
\begin{equation}\label{EQN.GapsBad2}
|n_i'-n_i|\le 2r\cdot \mathrm{diam}(\Omega').
\end{equation}
The assumption on $(\beta_{ij})_{j=1}^d$ guarantees that there exists
a constant $c_5>0$ with the property that there at most $c_5$ values
of $n_i'-n_i$ for which (\ref{EQN.GapsBad1}) and (\ref{EQN.GapsBad2})
hold. For $r$ sufficiently large, each such value uniquely determines
the value of $a_i-a_i'$ satisfying (\ref{EQN.GapsBad1}). Therefore the
volume of $U$ can take one of at most $c_4=c_5^{k-d}$ values.
\end{proof}
Finally, following (P4), let $G=\mathrm{GL}_{k-d}(\Q)\cap M_{k-d}(\Z)$
and consider the set of $\alpha\in\R^{d(k-d)}$ with the property that,
for all $B\in G$, the point $\beta=B^{-1}\alpha$ satisfies the
hypotheses of Lemma \ref{LEM.Reg1->Reg2ForBad}. For $1\le \ell\le k-d$
define $\mc{A}_\ell\subseteq\R^{d(k-d)}$ by
\[\mc{A}_\ell=\{(\beta_{ij})\in\R^{d(k-d)}:(\beta_{\ell j})_{j=1}^d\in\B_{d,1}\}.\]
By Theorem \ref{THM.HAWForBad}, $\B_{d,1}$ is an HAW set, and it
follows from Lemma \ref{LEM.HAWDirectSum} that each set $\mc{A}_\ell$
is also HAW. Therefore the set
\[\bigcap_{\ell=1}^{k-d}\mc{A}_\ell\]
is also HAW, as is the set $\A$ defined by
\[\A:=\bigcap_{B\in G}B\left(\bigcap_{\ell=1}^{k-d}\mc{A}_\ell\right).\]
We conclude that the set $\mc{A}$ has Hausdorff dimension $d(k-d),$
and the conclusion of Theorem \ref{THM.FreqsUppBds1} holds for any
$\alpha\in\mc{A}$.

\section{Proof of Theorem \ref{THM.FreqsUppBds2}}
In this section we use $w_i, B, \beta, L_i', \mathrm{sing}_2'(r),$ and $\mathrm{reg}_2'(r)$ to denote the
quantities introduced in the previous section. Many
of the ideas in the proof of Theorem \ref{THM.FreqsUppBds1} will be
applied directly in our proof of Theorem \ref{THM.FreqsUppBds2}. Since
we will be working with a set consisting of a.e. $\alpha$, the
arguments that we included in principle (P4) above actually become
easier, although we lose a power of a logarithm. We no longer need
HAW, and things are better for us in that respect. Most of the proof
of the bound for $\#\xi_2(r)$ in Theorem \ref{THM.FreqsUppBds2}
follows quite easily from Lemma \ref{LEM.FreqsVols2} and a
modification of the argument used to prove Lemma
\ref{LEM.VolumeBdForBAD}. However, there is an important difference
when we come to the proof of the bound for $\#\xi_1(r)$, which is the
reason why we need a slightly stronger hypothesis.

Before going into the details of the proof let us explain why our
previous proof does not give a non-trivial upper bound for
$\#\xi_1(r)$. We can still apply what we called principle (P1) before,
to draw some conclusions about the density modulo $1$ of the values
taken by $L_i'(n)$, for $|n|<r$. However, since we will be working
with a set containing almost every $\alpha\in\R^{d(k-d)},$ we no
longer have the badly approximable hypothesis. This causes enough loss
in the arguments of the previous proofs to lead to upper bounds for
$\#\xi_1(r)$ which are worse than the trivial bounds obtained from
counting connected components. In slightly more detail, note that what
we can say for a.e. $\alpha$ is determined by Theorem
\ref{THM.KhinGros} and the Borel-Cantelli Lemma. For example, supposing that $B$ is fixed, it is
true that for a.e. $\alpha$, and for any $\epsilon>0,$ there is a
constant $C>0$ such that, for each $1\le i\le k-d$,
\[\|L_i'(n)\|\ge \frac{C}{|n|^d(\log
  (2+|n|))^{1+\epsilon}}\quad\text{for all}~n\in\Z^k\setminus\{0\}.\]
Following the argument of Lemma \ref{LEM.Reg1->Reg2ForBad}, we can
apply Theorem \ref{THM.TransferTheorem} to deduce that there is a
constant $c_3>0$ such that every connected component of
$\mathrm{reg}_2'(r)$ is contained in a cube in $F$ of side length
no greater than a constant times
\[\frac{(\log r)^{d(1+\epsilon)}}{r^d}.\]
However when we apply the next step of the argument we reach the
conclusion that every connected component of $\mathrm{reg}_2(r)$
extends to at most a constant multiple of
\[R(r):=(\log r)^{(k-d)(d+1)(1+\epsilon)}\] connected components of
$\mathrm{reg}_2(r+c_2)$. The argument in the proof of Lemma
\ref{LEM.VolumeBdForBAD} tells us that the number of possible volumes
of connected components of $\mathrm{reg}_2(r)$ is also bounded by a
constant times $R(r)$. In the end this gives the bound
\[\#\xi_1(r)\le cR(r)^{R(r)},\] which tends to infinity faster than
any power of $r$. This is much larger than the number of equivalence
classes of patches of size $r$, which in our case is only at
most a constant times $r^{d(k-d)}$.

In order to work around the problem described in the previous
paragraph, when dealing with $\xi_1(r)$ we will make use of the extra
hypothesis in Theorem \ref{THM.FreqsUppBds2} on the shape of
$\Omega$. We do not consider patches of arbitrary shape, but instead
restrict our attention to hypercubes.
This allows us to say something more precise
about the set $\mathrm{sing}_2(r+2c_2)\setminus\mathrm{sing}_2(r)$, as
demonstrated by the following lemma.
\begin{lemma}\label{LEM.Reg1->Reg2ForAE}
  Assume that $\Omega'$ is a hypercube $(-1,1)^{d}$,
and let
  $\psi:\R\rar [0,\infty)$ be a non-increasing function. Suppose that,
  for every $1\le i\le k-d,$
\begin{equation}\label{EQN.aeHyp1}
\|L_i'(n)\|\ge \frac{\psi (|n|)}{|n|^d}\quad\text{for all}~n\in\Z^k\setminus\{0\},
\end{equation}
and that, for every $1\le i\le k-d$ and $1\le j'\le d,$
\begin{equation}\label{EQN.aeHyp2}
\left\|\sum_{\substack{j=1\\j\not= j'}}^{d}n_j\beta_{ij}\right\|\ge \frac{3^d}{|n|^d\psi(|n|)^d}\quad\text{for all}~n\in\Z^k\setminus\{0\}.
\end{equation}
Then there exist $c_1,c_2>0$ such that every element of $\xi_1(r)$ is a sum of at most $c_1$ volumes of connected components of $\mathrm{reg}_2(r+c_2)$, divided by the volume of $\mc{W}$.
\end{lemma}
\begin{proof}
  Much of this proof follows the proof of Lemma
  \ref{LEM.Reg1->Reg2ForBad}, and we use the notation developed in the
  proof of that lemma. First, as before, there is a constant $c_2$
  such that (\ref{EQN.PatchInclusions}) holds for all $y\in Y$ and all
  sufficiently large $r$. By the same argument as before, we make use
  of (\ref{EQN.aeHyp1}) and apply Theorem \ref{THM.TransferTheorem},
  with $X$ chosen so that $X=2\psi(X)r$, to conclude that if $U$ is
  any connected component of $\mathrm{reg}_2'(r)$ then
\[U=\{t\in W: \ell_i<t_i<r_i\},\]
with
\[r_i-\ell_i\le\frac{1}{m_i(2r\psi (r))^d}\quad\text{for each}\quad 1\le i\le k-d.\]

Next we count the number of translates of hyperplanes corresponding to
$\mathrm{sing}_2'(r+2c_2)\setminus\mathrm{sing}_2'(r)$, which intersect
$U$. Each such affine hyperplane has the form
\begin{equation}\label{EQN.SingHyperplane}
H_i+L_i'(n)e_{d+i}+m_i^{-1}e_{d+i}\Z,
\end{equation}
for some $1\le i\le k-d$ and $n\in\Z^d$ with $r\le |n|<r+2c_2$. This
is where we begin making use of our hypothesis on $\Omega'$.

Suppose that $\mc{N}\subseteq\Z^d$ is a collection of integer vectors
$n$ with $r\le |n|<r+2c_2$. If $|\mc{N}|>4c_2(k-d)$ then there are
$n,n'\in\mc{N}$, and an index $1\le j'\le d$, for which
$n_j=n_j'$. Then, for any $1\le i\le k-d$ we have by
(\ref{EQN.aeHyp2}) that, for all sufficiently large $r$,
\[L_i'(n-n')\ge \frac{3^d}{(2r+4c_2)^d\psi (2r+4c_2)^d}\ge
\frac{1}{r^d\psi (r)^d}.\] It follows that, for each $i$, the number
of integers $n\in\Z^d$ with $r\le |n|<r+2c_2$, for which the
hyperplane (\ref{EQN.SingHyperplane}) intersects $U$, is bounded
by a constant that only depends on $\beta$.  The conclusion of the
lemma now follows in the same way as that of Lemma
\ref{LEM.Reg1->Reg2ForBad}.
\end{proof}
\begin{lemma}\label{LEM.VolumeBdForAE}
  Let $\psi:\R\rar [0,\infty)$ be a non-increasing function and
  suppose that (\ref{EQN.aeHyp1}) holds for all $1\le i\le k-d$. Then,
  even without the additional hypothesis on
  $\Omega'$, there is a constant $c_4>0$ such that, for any $r>0$,
  there are at most
\[c_4\psi(r)^{-(d+1)(k-d)}\]
distinct volumes of connected components of $\mathrm{reg}_2(r)$.
\end{lemma}
\begin{proof}
  In the proof of the previous lemma we showed that, under hypothesis
  (\ref{EQN.aeHyp1}), every connected component of
  $\mathrm{reg}_2'(r)$ is contained in a box of side length at
  most a constant times
\[\frac{1}{r^d\psi (r)^d}.\]
The rest of the proof follows by using (\ref{EQN.aeHyp1}) again,
together with the same argument used to prove Lemma
\ref{LEM.VolumeBdForBAD}.
\end{proof}
To finish the proof of Theorem \ref{THM.FreqsUppBds2} we apply the Borel-Cantelli Lemma (see the paragraph before Theorem
\ref{THM.KhinGros}) to each of the individual linear forms $L_i'$ in
$d$-variables and then, for each $i$ and $j'$, to each of the linear
forms on the left of (\ref{EQN.aeHyp2}) in $(d-1)$-variables. In this
way we find that if we take $\psi(r)=(\log r)^{-1-\epsilon}$ then, for
each $B$, hypotheses (\ref{EQN.aeHyp1}) and (\ref{EQN.aeHyp2}) are
satisfied for almost every choice of $\beta$. If we take the image of
this full measure set of $\beta$ under the non-degenerate linear map
$B$ we obtain a full measure set of $\alpha$. Intersecting over all
choices for $B$ still leaves us with a full measure set of $\alpha$ to
which Lemmas \ref{LEM.Reg1->Reg2ForAE} and \ref{LEM.VolumeBdForAE}
apply, giving us the bounds in Theorem \ref{THM.FreqsUppBds2}.

Finally note that for the first part of Theorem
\ref{THM.FreqsUppBds2}, the bound for $\xi_2(r)$, we do not need Lemma
\ref{LEM.Reg1->Reg2ForAE} and can instead appeal directly to Lemmas
\ref{LEM.FreqsVols2} and \ref{LEM.VolumeBdForAE}. Therefore for that
bound the additional hypothesis on $\Omega$ is not necessary, and
neither for that matter is hypothesis (\ref{EQN.aeHyp2}).

\section{Proof of Theorem \ref{THM.FreqsLowBds}}
In this section we turn to the problem of finding lower bounds for the number of distinct frequencies. For simplicity we consider only type $2$ patches, and we specialize to the case when $\mc{W}$ is the parallelotope in $F$ generated by the standard basis vectors $e_{d+1},\ldots , e_k$.

By Lemma \ref{LEM.FreqsVols2} we know that $\#\xi_2(r)$ is equal to the number of distinct volumes of $\mathrm{reg}_2(r)$. For $r>0$, define
\begin{equation*}
\mc{A}(r):=\{L_1(n)~\mathrm{mod}~1:n\in r\Omega'\cap\Z^d\}\subseteq\R/\Z,
\end{equation*}
where $\Omega'$ is defined as in the proof of Lemma \ref{LEM.Reg1->Reg2ForBad}. Thinking of the elements of $\mc{A}(r)$ as being arranged in order on the circle, the number of distinct distances between consecutive elements gives a lower bound for $\#\xi_2(r)$. This follows from formula (\ref{EQN.GenericComponent}) for the generic connected component of $\mathrm{reg}_2(r)$, by fixing $i=2,\ldots , k-d$ and then letting the pair $(n_1,n_1')$ run over all choices which give consecutive elements of $\mc{A}(r)$. Therefore, in proving Theorem \ref{THM.FreqsLowBds} we will consider one linear form in $d$ variables, effectively concentrating on $L_1$ and ignoring the variables of $\alpha$ corresponding to the other linear forms. With this strategy in mind, we present the following lemma, the proof of which is a modification of the argument used to prove \cite[Lemma 6.1]{BlehHommJiRoedShen2012}.
\begin{lemma}\label{LEM.InfGaps}
Suppose $d\ge 2,$ let $\epsilon>0$, and, using the notation introduced in Section \ref{SEC.Diophantine}, suppose that $\alpha=(\alpha_1,\ldots ,\alpha_d)\in\R^d$ satisfies the following hypotheses:
\begin{enumerate}
\item[(i)] $(\alpha_1,\ldots ,\alpha_{d-1})\in\mc{B}_{d-1,1}$,\vspace*{.05in}
\item[(ii)] $\alpha_d/\alpha_1\in\mc{E}(q^{-(2d-1+\epsilon)})\cap (0,1/2),$ and\vspace*{.05in}
\item[(iii)] the collection of numbers $\{1,\alpha_1,\ldots ,\alpha_d\}$ is $\Q$-linearly independent.
\end{enumerate}
Then, taking $L_1$ to be the linear form determined by $\alpha$, and $\Omega'$ to be the region
\[\Omega'=\{x\in (0,1]^d : x_d+2x_1\le 1\},\]
there is a constant $C>0$ such that, for a set of numbers $r>0$ which can be taken arbitrarily large, the number of distinct distances between consecutive elements of $\mc{A}(r)$ is greater than $Cr^{\epsilon/(d+\epsilon)}$.
\end{lemma}
\begin{proof}
By hypothesis (ii), we can choose $p,q\in\N$, coprime and arbitrarily large, so that $p/q<1/2$ and
\[\alpha_d=p\left(\frac{\alpha_1}{q}\right)+\gamma,\quad\text{with}\quad |\gamma|\le\frac{\alpha_1}{q^{2d+\epsilon}}.\]
We consider only the case when $\gamma>0$, as the proof of the other case is similar. We have, for $r>0$, that
\[\mc{A}(r) = \left\{\ell \left(\frac{\alpha_1}{q}\right)+\sum_{i=2}^{d-1}n_i\alpha_i+n_d\gamma : \ell=n_1q+n_dp,~ n\in r\Omega'\cap \mathbb{Z}^d\right\}.\]
In order to understand how the elements of $\mc{A}(r)$ are ordered on the circle, first let $\mc{M}(r)\subseteq\Z^{d-1}$ be the collection of integer vectors $m$ which can be written in the form
\[m=(n_1q+n_dp, n_2,\ldots ,n_{d-1})\quad\text{for some}\quad n\in r\Omega'\cap\Z^d.\]
Label this set as $\mc{M}(r)=\{m^{(i)}\}_{i=1}^M$, so that
\begin{equation}\label{EQN.KLabel}
\left\{m^{(i)}_1\left(\frac{\alpha_1}{q}\right)+\sum_{j=2}^{d-1}m^{(i)}_j\alpha_j\right\} < \left\{m^{(i+1)}_1\left(\frac{\alpha_1}{q}\right)+\sum_{j=2}^{d-1}m^{(i+1)}_j\alpha_j\right\},
\end{equation}
for each $1\le i<M$ (where $\{\cdot\}$ denotes the fractional part). Where convenient below, we will identify $m^{(M+1)}$ with $m^{(1)}$. By factoring out $1/q$ and applying hypothesis (i), we see that the difference between the right and left hand sides of the previous equation is greater than
\[\frac{c}{r^{d-1}q^d},\]
for some constant $c>0$. On the other hand,
\[|n_d\gamma|\le\frac{\alpha_1r}{q^{2d+\epsilon}},\]
and this motivates us to select $r=(c/\alpha_1)^{1/d}q^{1+\epsilon/d}$. With this choice of $r$ it follows that, for each $1\le i<M$,  as $n$ runs over all elements of $r\Omega'\cap\Z^d$ which correspond to $m^{(i)}\in\mc{M}$, the values taken by $\{L_1(n)\}$ fall between the numbers on the left and right hand sides of (\ref{EQN.KLabel}). Since every value of $L_1(n)$ is uniquely determined by specifying the corresponding values of $m^{(i)}$ and $n_d$, we see that the ordering of the fractional parts of $L_1(n)$ is given by the lexicographic ordering of the pairs $(m^{(i)},n_d)$, first by $i$ and then by $n_d$.

For each integer $1\le i\le M$, let $a_i$ and $B_i$ be positive integers satisfying
\[m^{(i)}_1=a_i q+B_i p\quad\text{and}\quad B_i+2a_i\le r,\]
with $B_i$ taken as large as possible (obviously, this choice is unique). Then, using the fact that $p/q<1/2$, we have that
\begin{align*}
&\{(n_1,n_d) \in \mathbb{N}^2 :  n_1q+n_dp=m^{(i)}_1,~n_d+2n_1\le r\} \cr
&\qquad\qquad=\{(a_i+jp,B_i-jq) :  0\le j< B_i/q\}.
\end{align*}
We use this observation to identify the collection of values of $n_d$ corresponding to a given value of $m^{(i)}_1$.
Now let $b_i$ be the integer with the property that $(m^{(i)},B_i)$ and $(m^{(i+1)},b_i)$ are consecutive in the lexicographic ordering. It follows from the formula above that $0<b_i\le q$.

Finally, for each integer $1\le k\le r/q$ choose an integer $1\le i_k\le M$ with $m^{(i)}_1=kqp$. Such an integer always exists, as can be seen by taking $n_1=0$ and $n_d=kq$, and it is also clear that $B_{i_k}=kq$. Let $n^{(k)}$ and $n^{(k)'}$ be the elements of $r\Omega'\cap\Z^d$ corresponding to $(m^{(i_k)},B_{i_k})$ and $(m^{(i_k+1)},b_{i_k})$, respectively, and set
\begin{align*}
\Delta_k=\left\{L_1(n^{(k)'})-L_1(n^{(k)})\right\}.
\end{align*}
We have that
\begin{align*}
\Delta_k=\left\{\sum_{i=1}^{d-1}m_i\alpha_i-(B_{i_k}-b_{i_k})\alpha_d\right\},
\end{align*}
for some integers $m_i$. Now the key observation is that
\[(k-1)q\le B_{i_k}-b_{i_k}<kq,\]
so that different choices of $k$ necessarily give different values of $B_{i_k}-b_{i_k}$. By hypothesis (iii) this implies that different choices of $k$ also give different values of $\Delta_k$. Therefore, taking $q$ sufficiently large, we find that the number of distinct distances between consecutive elements of $\mc{A}(r)$ is at least
\[\frac{r}{q}\ge Cr^{\epsilon/(d+\epsilon)},\]
for some constant $C>0$.
\end{proof}
We leave it to the reader to check that the proof of the preceding lemma can be modified to deal with any translate of the region $\Omega'$ in its statement.

To complete the proof of Theorem \ref{THM.FreqsLowBds}, consider the set of points $\alpha\in\R^{d(k-d)}$ for which $(\alpha_{11},\alpha_{12},\ldots ,\alpha_{1d})$ satisfies the hypotheses of Lemma \ref{LEM.InfGaps}. The set of points $(\alpha_{11},\ldots ,\alpha_{1(d-1)})$ in $\R^{d-1}$ which lie in $\mc{B}_{d-1,1}$ has Hausdorff dimension $d-1$. For every point in this set, the set of real numbers $\alpha_{1d}$ for which \[\alpha_{1d}/\alpha_{11}\in\mc{E}(q^{-(2d-1+\epsilon)})\cap (0,1/2)\]
has Hausdorff dimension $2/(2d+\epsilon)$. The latter claim follows from the well known Jarn\'{i}k-Besicovitch Theorem, a statement of which can be found in \cite[Section 12]{BereDickVela2006}. Finally, the collection of points which do not satisfy hypothesis (iii) in Lemma \ref{LEM.InfGaps} is countable. Using basic properties of Hausdorff dimension (see \cite[Exercise 7.9]{Falc2003}), we conclude that, for each choice of $\epsilon>0,$ the collection of points $\alpha\in\R^{d(k-d)}$ to which our lemma can be applied has Hausdorff dimension
\[d(k-d-1)+(d-1)+2/(2d+\epsilon).\]
Taking the union of these sets over all $\epsilon>0$ gives a set of dimension $d(k-d)+(1/d-1)$. For each point in this set we can choose $\Omega$ so that the corresponding set $\Omega'$ is a translate of the set in the statement of Lemma \ref{LEM.InfGaps}. Then, by the discussion at the beginning of this section, we obtain the statement in the conclusion of Theorem \ref{THM.FreqsLowBds}.

We point out that the proof above can be modified to give faster growth rates for $\#\xi_2(r)$, on a set of smaller Hausdorff dimension.

\appendix
\section{Topological proof of the Steinhaus Theorem}
The Steinhaus Theorem, also known as the Three Distances Theorem, is the statement that, for any $\alpha\in\R$ and for any $N\in\N$, removing the collection of points
\[\{n\alpha~\mathrm{mod}~1:1\le n\le N\}\subseteq\R/\Z\]
partitions $\R/\Z$ into component intervals of at most $3$ distinct lengths, and if there are $3$ lengths then one is the sum of the other two. There are many proofs of this theorem. Here we offer our own which, although not the most direct, illustrates how the topological and dynamical ideas discussed in this paper can be brought to bear on such problems. The idea of our proof also applies to other Steinhaus-type problems (for example, those discussed in \cite{CobeGrozVajaZaha2002}, \cite{FraeHolz1995}, and \cite{GeelSimp1993}).

Let $\alpha$ and $N$ be given and let $\mc{J}_N$ be the collection of component intervals of the set
\[\R/\Z\setminus\{n\alpha~\mathrm{mod}~1:1\le n\le N\}.\]
Let $T_\alpha:\R/\Z\rar\R/\Z$ be rotation by $\alpha$, defined by $T_\alpha(x)=x+\alpha$, and define a CW-complex $\Gamma_N$ as follows:
\begin{enumerate}
\item[(i)] $0$-cells of $\Gamma_N$ are indexed by elements of $\mc{J}_N$, and
\item[(ii)] For $J,J'\in\mc{J}_N$, there is a $1$-cell from $J$ to $J'$ if and only if $T_\alpha(J)\cap J'\not=\emptyset$.
\end{enumerate}
For any $0$-cell $J$, if $-(N+1)\alpha\in J$ then $T_\alpha(J)$ intersects two elements of $\mc{J}_N.$ Otherwise $T_\alpha (J)$ intersects one element of $\mc{J}_N$. Similarly, if $(N+1)\alpha\in J$ then $T^{-1}_\alpha(J)$ intersects two elements of $\mc{J}_N.$ Otherwise $T^{-1}_\alpha (J)$ intersects one element of $\mc{J}_N$. It follows that $\Gamma_N$ is homotopic either to a circle  or to a bouquet of two circles, and in particular that it takes one of the three forms illustrated in Figure 4.

\begin{figure}[h]
\centering
\def\svgwidth{0.60\columnwidth}
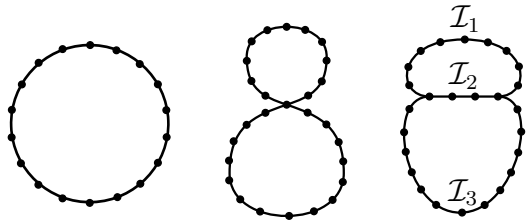
\caption{Possible shapes for $\Gamma_N$.}
\end{figure}

The leftmost diagram in Figure 4 can only occur if $\alpha$ is rational and $N$ is greater than or equal to the order of $\alpha$ in $\Q/\Z$. The middle diagram represents the `case of two distances' in the Three Distances Theorem, which can occur for any $\alpha$ but depends on $N$, and the right hand diagram is the `generic' case for irrational $\alpha$.

For simplicity let us assume that $\alpha$ is irrational and that we are in the rightmost case of Figure 4 (the other cases follow from similar arguments). Choose any point $\beta\in\R$ which is not in the set $\alpha\Z+\Z$, and for each $J\in\mc{J}_N$ let
\[\xi_J=\lim_{R\rar\infty}\frac{1}{R}\sum_{1\le i\le R}\chi_J\left(T^i_\alpha(\beta)\right),\]
where $\chi_J$ denotes the indicator function of $J\subseteq\R/\Z$. Whatever the choice of $\beta$, it follows from the unique ergodicity of the map $T_\alpha$ that $\xi_J=|J|$. Note that, even if we didn't have unique ergodicity, we could still apply the Birkhoff Ergodic Theorem to a generic point $\beta$ to draw the same conclusion.

Finally, referring again to Figure 4, it is clear that intervals $J$ corresponding to $0$-cells on the same line segment $\mc{I}_i$, for $i=1,2,$ or $3$, give the same values of $\xi_J$. Therefore there are at most $3$ possible values for the lengths $|J|$. From Figure 4 it is also clear that the length corresponding to $\mc{I}_2$ is the sum  of the lengths corresponding to $\mc{I}_1$ and $\mc{I}_3$.

\vspace{.1in}

{\footnotesize
\noindent AH, HK, JW\,:\\
Department of Mathematics, University of York,\\
Heslington, York, YO10 5DD, England\\
alan.haynes@york.ac.uk\\henna.koivusalo@york.ac.uk\\jamie.walton@york.ac.uk

\vspace{.1in}

\noindent LS\,:\\
Department of Mathematics, University of Texas at Austin,\\
Austin, TX 78712-1082, USA\\
sadun@math.utexas.edu

}

\end{document}